\documentclass[a4paper]{amsart}
\usepackage[dvips]{graphics}
\usepackage[english]{babel}
\usepackage{amsfonts}
\usepackage{amsthm}
\usepackage{amsmath}
\usepackage{amscd}
\usepackage[active]{srcltx} 
\usepackage{latexsym}
\usepackage{amssymb}
\usepackage[latin1]{inputenc}
\usepackage{xcolor}
\usepackage{enumitem}
\usepackage{caption, subcaption}

\numberwithin{equation}{section}

\newtheorem{thm}{Theorem}[section]
\newtheorem{prop}[thm]{Proposition}

\newtheorem{cor}[thm]{Corollary}
\newtheorem*{cor*}{Corollary}
\newtheorem{lema}[thm]{Lemma}
\newtheorem*{lema*}{Lemma}
\newtheorem*{hyp}{Hypotheses}

\theoremstyle{definition}

\newtheorem*{prob*}{Problem}

\newtheorem{Def}[thm]{Definition}

\newtheorem{example}[thm]{Example}
\newtheorem{obs}[thm]{Remark}

\newtheorem*{obs*}{Remark}
\newtheorem*{thm*}{Theorem}
\newtheorem*{prop*}{Proposition}

\newcommand{\PI}[2]{\left\langle \,#1 , #2\, \right\rangle}

\newcommand{\PIB}[2]{\Big\langle \,#1 , #2\, \Big\rangle}

\newcommand{\set}[1]{\left\{ \,#1\, \right\}}

\newcommand{\x}{\times}

\newcommand{\CC}{\mathbb{C}}
\newcommand{\RR}{\mathbb{R}}
\newcommand{\NN}{\mathbb{N}}

\newcommand{\St}{\mathcal{S}}

\newcommand{\HH}{\mathcal{H}}

\newcommand{\PP}{\mathcal{P}}

\newcommand{\mc}[1]{\mathcal{#1}}

\newcommand{\noi}{\noindent}

\newcommand{\la}{\lambda}

\makeatletter
\newcommand{\eqnum}{\refstepcounter{equation}\textup{\tagform@{\theequation}}}
\makeatother

\DeclareMathOperator{\real}{Re}

\begin{document}

\title{Krein-\v{S}mul'jan Theorem Revisited}

\author[S.~Gonzalez Zerbo]{Santiago Gonzalez Zerbo}
\address{Instituto Argentino de Matem\'{a}\-tica ``Alberto P. Calder\'{o}n'' (CONICET), Saavedra 15 (1083) Buenos Aires, Argentina}
\email{sgzerbo@fi.uba.ar} 

\author[A.~Maestripieri]{Alejandra Maestripieri}
\address{Instituto Argentino de Matem\'{a}\-tica ``Alberto P. Calder\'{o}n'' (CONICET), Saavedra 15 (1083) Buenos Aires, Argentina}
\email{amaestri@fi.uba.ar}
 
\author[F.~Mart\'{\i}nez Per\'{\i}a]{Francisco Mart\'{\i}nez Per\'{\i}a}
\address{Centro de Matem\'{a}tica de La Plata (CMaLP) -- FCE-UNLP, La Plata, Argentina \\
and Instituto Argentino de Matem\'{a}tica ``Alberto P. Calder\'{o}n'' (CONICET), Saavedra 15 (1083) Buenos Aires, Argentina}
\email{francisco@mate.unlp.edu.ar}


\begin{abstract}
We present a generalization of Krein-\v{S}mul'jan theorem which involves several operators. Given bounded selfadjoint operators $A,B_1,\ldots,B_m$ acting on a Hilbert space $\HH$, we provide sufficient conditions to determine whether there are $\la_1,\ldots,\la_m\in \RR$ such that $A + \sum_{i=1}^m \la_i B_i$ is a positive semidefinite operator.
\end{abstract}

\maketitle

\section{Introduction}

Along this paper $(\HH,\PI{\cdot}{\cdot})$ denotes a complex Hilbert space, and $\mc{L}(\HH)$ stands for the algebra of bounded linear operators in $\HH$.
An operator $A\in \mc{L}(\HH)$ is {\it positive semidefinite} if $\PI{Ax}{x}\geq 0$ for all $x\in\HH$; and it is {\it positive definite} if there exists
$\alpha>0$ such that $\PI{Ax}{x}\geq \alpha\|x\|^2$ for every $x\in\HH$. 

Given bounded selfadjoint operators $A,B_1,\ldots,B_m$ acting on $\HH$, the aim of this work is to determine whether there are $\la_1,\ldots,\la_m\in \RR$ such that the operator $A + \sum_{i=1}^m \la_i B_i$ is positive semidefinite. If $\geq$ denotes L\"owner's partial order of selfadjoint operators, the problem can be restated as whether the inequality
\begin{equation}\label{LMI}
A + \sum_{i=1}^m \la_i B_i \geq 0
\end{equation}
is feasible. If $\HH$ is finite dimensional this is known as a \emph{linear matrix inequality} (LMI), an area which has been thoroughly studied since the 1940's for its applications in System and Control theory, see \cite{Boyd} and the references therein.

Another reason that makes this problem interesting is that it is closely related to the existence of minimizers for quadratically constrained quadratic programming (QCQP) problems. 
A QCQP problem can be posed as:
\begin{alignat*}{3}
& \text{minimize}   \quad && f(x)&&=\PI{Ax}{x}+2\real\PI{y_0}{x} + \alpha_0\\
& \text{subject to}   \quad && g_i(x)&&=\PI{B_i x}{x}+2\real\PI{y_i}{x}\leq\alpha_i,\qquad i=1,\ldots,m,
\end{alignat*}
where the optimization variable $x$ varies in $\HH$, and the data consists of bounded selfadjoint operators
$A,B_1,\ldots,B_m$ acting in $\HH$, vectors $y_i\in\HH$ and scalars $\alpha_i\in\RR$, for $i=0,1,...,m$. 
Note that the Hessian of such a quadratic function is constant. In particular, the Hessian of $f,g_1,\ldots,g_m$ are given by the selfadjoint operators $A,B_1,\ldots,B_m$, respectively. Hence, if $x_0$ is a minimizer of the above problem then there exist $\la_1,\ldots,\la_m\in\RR$ such that \eqref{LMI} holds, see e.g. \cite{Luenberger,Marsden}.

For a finite dimensional space, studies on the simplest case (i.e. $m=1$) can be traced back to works of Finsler \cite{Finsler}, Hestenes \cite{Hestenes51}, and Calabi \cite{Calabi}. But the result characterizing the feasibility of $A + \la B\geq 0$ in an arbitrary Hilbert space is known as the Krein-Smul'jan theorem \cite{KSruss,KS}, see also \cite{Kuhne64, Kuhne68}. Given a selfadjoint operator $B\in\mc{L}(\HH)$ we say that $B$ is indefinite if it is not semidefinite i.e. there exist $x_+,x_-\in\HH$ such that $\PI{Bx_+}{x_+}>0$ and $\PI{Bx_-}{x_-}<0$.

\begin{thm}\label{teo_krein_smuljan_azizov}
If $B\in\mc{L}(\HH)$ is indefinite, then there exists $\la\in\RR$ such that \\ $A + \la B\geq 0$ if and only if
\[
\PI{Ax}{x}\geq 0 \qquad \text{whenever}\qquad \PI{Bx}{x}=0. 
\]
In this case, 
\[
\frac{\PI{Ay}{y}}{\PI{By}{y}}\leq \frac{\PI{Az}{z}}{\PI{Bz}{z}}
\]
 for every $y, z\in \HH$ such that $\PI{By}{y}<0$ and $\PI{Bz}{z}>0$. Also, if
\begin{equation}\label{eq:def_lambdas}
\la_-:=-\inf_{\PI{Bx}{x}>0}\frac{\PI{Ax}{x}}{\PI{Bx}{x}}\quad\quad\text{and}\quad\quad \la_+:=-\sup_{\PI{Bx}{x}<0}\frac{\PI{Ax}{x}}{\PI{Bx}{x}},
\end{equation}
then $\lambda_-\leq\lambda_+$ and $\{\la\in\RR: \ A + \la B\geq 0 \}=[\lambda_-,\lambda_+]$.
\end{thm}


To the best of our knowledge, there is no such a result for an inequality which involves several variables like \eqref{LMI}. Even in the finite dimensional setting, there are only a few results.
Among them, it is worthwhile mentioning the works by Dines \cite{Dines42,Dines43} and Hestenes and McShane \cite{Hestenes}. 

\medskip

The paper is organized as follows. Section \ref{Indef} starts with a discussion about weakly indefinite sets of selfadjoint operators. We show that this notion is only sufficient to prove a generalization of Krein-Smul'jan theorem in the case of pairs $\{B_1,B_2\}$. For finite sets $\{B_1,\ldots,B_m\}$ with $m>2$ it is necessary to impose some extra condition, named strongly indefiniteness. After discussing what strongly indefiniteness means, in Theorem \ref{prop_ks_strongly} we state a generalization of Krein-Smul'jan theorem. Finally, in Section \ref{Hes-McShane} we give a sufficient condition on $\{B_1,\ldots,B_m\}$ to be strongly indefinite, which is inspired by the results of Hestenes and McShane in \cite{Hestenes}.

\section{Weakly indefinite sets of selfadjoint operators}\label{Indef}

We start with a definition which is mainly motivated by \cite{Dines43,Hestenes}.

\begin{Def} 
A set of selfadjoint operators $\{B_1,\ldots,B_m\}$ is {\it weakly indefinite} if
\[
\sum_{i=1}^m\mu_iB_i\text{ is indefinite for every $(\mu_1,\ldots,\mu_m)\in\RR^m\setminus\{0\}$}.
\]
\end{Def}

If $\{B_1,B_2,\ldots,B_m\}$ is weakly indefinite, then any subset of it is also weakly indefinite. In particular, $B_i$ is indefinite for every $i=1,2,\ldots,m$. Also, if $\{B_1,\ldots,B_m\}$ is weakly indefinite then it is a linearly independent set.

\medskip

Given a selfadjoint operator $B\in\mc{L}(\HH)$, denote by $Q(B)$ the set of neutral vectors for the quadratic form induced by $B$, $Q(B)=\{x\in\HH: \ \PI{Bx}{x}=0\}$. 
Given a set of selfadjoint operators $\{B_1,\ldots,B_m\}$ for brevity we write $Q_i=Q(B_i)$ for each $i=1,\ldots,m$. Also, we consider the sets of vectors which are positive (negative) with respect to the quadratic form induced by $B_i$:
\begin{align*}
\mc{P}_i^{+}=\set{x\in\HH\,:\,\PI{B_ix}{x}>0} \quad \text{and} \quad  
\mc{P}_i^-=\set{x\in\HH\,:\,\PI{B_ix}{x}<0}.    
\end{align*}

It is well known that $B$ is indefinite if and only if $Q(B)\setminus N(B)\neq \{0\}$, i.e. if
there exists $x\in\HH$ such that
\[
\PI{Bx}{x}=0\qquad\text{and}\qquad Bx\neq 0.
\]
The next result presents a sufficient condition to guarantee the weakly indefiniteness of $\{B_1,\ldots,B_m\}$.

\begin{prop} 
Given selfadjoint operators $B_1,\ldots,B_m\in\mc{L}(\HH)$, if there exists $x\in\HH$ such that
\[
x\in \textstyle{\bigcap_{j=1}^m}Q_j\quad\text{and}\quad
\{B_1x,\ldots, B_mx\} \ \text{is linearly independent in $\HH$}
\]
then $\{B_1,\ldots,B_m\}$ is weakly indefinite.
\end{prop}

\begin{proof} 
It suffices to show that $Q\big(\sum_{j=1}^m\la_jB_j\big)\setminus N\big(\sum_{j=1}^m\la_jB_j\big)\neq\{0\}$ for every $(\la_1,\ldots,\la_m)\in\RR^m\setminus\{ 0\}$. Given $(\la_1,\ldots,\la_m)\in\RR^m\setminus\{ 0\}$, note that $\big(\sum_{j=1}^m\la_jB_j\big)x$ is not trivial because $\{B_1x,\ldots, B_mx\}$ is linearly independent. Then, 
\[
x\in\big(\textstyle{\bigcap_{j=1}^m}Q_j\big)\setminus N\big(\sum_{j=1}^m\la_jB_j\big)\subseteq Q\big(\sum_{j=1}^m\la_jB_j\big)\setminus N\big(\sum_{j=1}^m\la_jB_j\big),
\]
and since $(\la_1,\ldots,\la_m)$ was arbitrary the proof is complete.
\end{proof}

Given $x\in\HH$, note that $\{B_1x,\ldots, B_mx\}$ is linearly independent if and only if $x\not\in N(\sum_{j=1}^m \la_jB_j)$ for every $(\la_1,\ldots,\la_m)\in\RR^m\setminus\{0\}$.

The sufficient condition presented above is not necessary to guarantee weakly indefiniteness of a set of operators, because it imposes that $\bigcap_{i=1}^m Q_i\neq\{0\}$. In Example \ref{conj indef pero no strongly} below we present a set $\{B_1,B_2,B_3,B_4\}$ which is weakly indefinite but $\bigcap_{i=1}^4 Q_i=\{0\}$.

\begin{lema}\label{prop_simetria} 
Given two indefinite selfadjoint operators $B_1,B_2\in\mc{L}(\HH)$, 
the family $\{B_1,B_2\}$ is weakly indefinite if and only if $B_i$ is indefinite in $Q_j$ for $j\neq i$.
In this case $Q_1\cap Q_2\neq\{0\}$.
\end{lema}

\begin{proof}
Assume, for example, that $B_1$ is indefinite in $Q_2$ but there exists $(\la_1,\la_2)\neq(0,0)$ such that $\la_1B_1 + \la_2B_2\geq 0$. If $\la_1=0$ then $\la_2B_2\geq 0$ leading to a contradiction. If $\la_1>0$ then $B_1 + \tfrac{\la_2}{\la_1}B_2\geq 0$. In  particular $B_1\geq 0$ in $Q_2$, which is a contradiction to our assumption. If $\la_1<0$ then it is easy to see that $B_1\leq 0$ in $Q_2$, which leads to another contradiction.

Conversely, 
suppose that $\{B_1,B_2\}$
is weakly indefinite and that $B_1$ is definite in $Q_2$. If $B_1\geq0$ in $Q_2$ then, by Theorem \ref{teo_krein_smuljan_azizov}, there exists $\la\in\RR$ such that $B_1+\la B_2\geq0$,
which is a contradiction to $\{B_1,B_2\}$ being indefinite. If $B_1\leq0$ in $Q_2$,
consider $-B_1$. By symmetry, $B_2$ is indefinite in $Q_1$.

\smallskip
To see that $Q_1\cap Q_2\neq\{0\}$, take $ y \in\mc{P}_1^-\cap Q_2$ and $ z \in\mc{P}_1^+\cap Q_2$, and choose $\theta\in[0,\pi)$ such that $\real\PI{B_2 y }{e^{i\theta} z }=0$.
Consider
\[
\gamma(t)=t\, y +(1-t)e^{i\theta}\, z ,\qquad t\in[0,1].
\]
Since $ y,z \in Q_2$, for $t\in[0,1]$
\[
\PI{B_2\gamma(t)}{\gamma(t)}=t^2\PI{B_2 y }{ y }+(1-t)^2\PI{B_2 z }{ z }+2\,t\,(1-t)\real\PI{B_2 y }{e^{i\theta} z }=0.
\]
Hence, $\gamma\big([0,1]\big)\subseteq Q_2$ and the real valued function 
\[
f(t)=\PI{B_1\gamma(t)}{\gamma(t)},\qquad t\in[0,1],
\]
satisfies $f(0)=\PI{B_1 y }{ y }<0$ and $f(1)=\PI{B_1 z }{ z }>0$. Thus, there exists $t_0\in(0,1)$ such that $f(t_0)=0$.
This implies that $\gamma(t_0)\in Q_1\cap Q_2$. 
Also, $\gamma(t_0)\neq 0$ because $\{y,z\}$ is a linearly independent set.
\end{proof}


In the following  we denote by $\Omega$ the feasibility set for inequality \eqref{LMI}, i.e.  
\begin{equation}\label{feasibility set}
\Omega=\Omega\big(A,(B_i)_{i=1}^m\big):=\set{(\la_1,\la_2,\ldots,\la_m)\in\RR^m\,:\, A+\sum_{i=1}^m\la_iB_i \geq 0  }.
\end{equation}
It is easy to check  that $\Omega$ is a closed convex subset of $\RR^m$. 

The next proposition characterizes the feasibility of \eqref{LMI} for $m=2$. Its proof follows the lines of one given in \cite{DM}.

\begin{thm}\label{prop_krein_smuljan_2} 
Given selfadjoint operators $A,B_1,B_2\in\mc{L}(\HH)$, assume that
$\{B_1,B_2\}$ is weakly indefinite. Then,
\[
A\geq0\quad\text{in}\quad Q_1\cap Q_2\qquad\text{if and only if}\qquad\Omega\neq\varnothing.
\]
\end{thm}

\begin{proof}
The fact that $\Omega\neq\varnothing$ trivially implies $A\geq0$ in $Q_1\cap Q_2$. To prove the converse, assume that $A\geq0$ in $Q_1\cap Q_2$.
By Lemma \ref{prop_simetria}, $B_1$ is indefinite in $Q_2$. Hence, fixing $y\in\mc{P}_1^-\cap Q_2$, $z\in\mc{P}_1^+\cap Q_2$ and choosing $\theta\in[0,\pi)$ so that $\real\PI{B_2y}{e^{i\theta}z}=0$, consider
\[
\gamma_\pm(t)=t\, y \pm(1-t)e^{i\theta}\, z ,\qquad t\in[0,1].
\]
Note that $\gamma_\pm\big([0,1]\big)\subseteq Q_2$ and 
%
take $t_\pm\in(0,1)$ as in the proof of Lemma \ref{prop_simetria} such that $\gamma_\pm(t_\pm)\in Q_1\cap Q_2$.
Now, we have the equations
\[
a\PI{B_1 y }{ y }+\tfrac{1}{a}\PI{B_1 z }{ z }+2\real\PI{B_1 y }{e^{i\theta} z }=\tfrac{1}{t_+(1-t_+)}\PI{B_1 \gamma_+(t_+)}{\gamma_+(t_+)}=0,
\]
\[
b\PI{B_1 y }{ y }+\tfrac{1}{b}\PI{B_1 z }{ z }-2\real\PI{B_1 y }{e^{i\theta} z }=\tfrac{1}{t_-(1-t_-)}\PI{B_1 \gamma_-(t_-)}{\gamma_-(t_-)}=0.
\]
where $a:=\frac{t_+}{1-t_+}$ and $b:=\frac{t_-}{1-t_-}$ are positive. Then, adding these two we get
\[
(a+b)\PI{B_1 y }{ y }+\big(\tfrac{1}{a}+\tfrac{1}{b}\big)\PI{B_1 z }{ z }=0,
\]
or equivalently,
\begin{equation}\label{eq_igualdad_conos}
a\,b = -\frac{\PI{B_1z}{z}}{\PI{B_1y}{y}}.
\end{equation}
Now, since $\PI{A\gamma_\pm(t_\pm)}{\gamma_\pm(t_\pm)}\geq0$, in the same fashion we get that
\[
0\leq \frac{1}{a\,b}\PI{Az}{z}+\PI{Ay}{y}.
\]
Combining this with \eqref{eq_igualdad_conos} yields
\[
\frac{\PI{A y }{ y }}{\PI{B_1 y }{ y }}\leq\frac{\PI{A z }{ z }}{\PI{B_1 z }{ z }},
\]
for arbitrary $y\in\mc{P}_1^-\cap Q_2$ and $z\in\mc{P}_1^+\cap Q_2$. Therefore,
\[
\sup_{y\in\mc{P}_1^-\cap Q_2}\frac{\PI{Ay}{y}}{\PI{B_1y}{y}}\leq\inf_{z\in\mc{P}_1^+\cap Q_2}\frac{\PI{Az}{z}}{\PI{B_1z}{z}}.
\]
If $\la_1\in\RR$ is such that $-\inf_{z\in\mc{P}_1^+\cap Q_2}\frac{\PI{Az}{z}}{\PI{B_1z}{z}}\leq\la_1\leq -\sup_{y\in\mc{P}_1^-\cap Q_2}\frac{\PI{Ay}{y}}{\PI{B_1y}{y}}$ then
\[
\PI{(A+\la_1 B_1)x}{x}\geq0\quad\text{for every $x\in(\mc{P}_1^-\cap Q_2)\cup (\mc{P}_1^+\cap Q_2)$.}
\]
Considering that $\PI{Ax}{x}\geq0$ for every $x\in Q_1\cap Q_2$, we then have that
\[
\PI{(A+\la_1 B_1)x}{x}\geq0 \quad\text{for every $x\in Q_2$}.
\]
Finally, by Theorem \ref{teo_krein_smuljan_azizov} there exists $\la_2\in\RR$ such that $A+\la_1B_1+\la_2 B_2\geq0$, i.e. $(\la_1,\la_2)\in\Omega$.
\end{proof}

\begin{cor}\label{prop_simetria_2} 
Given three indefinite selfadjoint operators $B_1,B_2, B_3\in\mc{L}(\HH)$, the family $\{B_1,B_2,B_3\}$
is weakly indefinite if and only $B_j$ is indefinite in $\bigcap_{i\neq j}Q_i$ for every $j=1,2,3$.
\end{cor}

\begin{proof}
It is analogous to the proof of Lemma \ref{prop_simetria}, using Theorem \ref{prop_krein_smuljan_2} instead of Theorem \ref{teo_krein_smuljan_azizov}.
%
\end{proof}

\section{Indefinite sets}

\begin{Def}\label{very indef}
Given selfadjoint operators $B_1,\ldots,B_m\in\mc{L}(\HH)$, $m\geq2$, the set $\{B_1,\ldots,B_m\}$ is {\it indefinite} if $B_j$ is indefinite in $\bigcap_{i\neq j} Q_i$ for every $j=1,\ldots,m$.
\end{Def}

Note that the above definition imposes that $\bigcap_{i\neq j} Q_i\neq\{0\}$ for any $j=1,\ldots,m$.

\begin{lema}\label{muy indef implica indef}
Assume that $\{B_1,\ldots,B_m\}$ is indefinite. Then, $\{B_1,\ldots,B_m\}$ is weakly indefinite.
\end{lema}

\begin{proof}
The proof is similar to that corresponding to Lemma \ref{prop_simetria}.
Suppose that there exists $\mu\in\RR^m\setminus\{0\}$ such that
$\sum_{i=1}^m\mu_iB_i\geq0$ and $\mu_j\neq0$ for some $j=1,\ldots,m$. If $\mu_j>0$ then $B_j+\sum_{i\neq j}\frac{\mu_i}{\mu_j}B_i\geq0$, so that $B_j\geq0$ in
$\bigcap_{i\neq j}Q_i$, leading to a contradiction. A similar argument holds if $\mu_j<0$.
\end{proof}



\begin{obs}
Consider an indefinite set $\{B_1,\ldots,B_m\}$. If $\bigcap_{i=1}^m Q_i=\{0\}$ then it is a maximal  indefinite set.

In fact, given any selfadjoint $B\in\mc{L}(\HH)$ such that $B\neq B_i$ for $i=1,\ldots,m$, by definition, a necessary condition for the set $\{B_1,\ldots,B_m,B\}$ to be indefinite is that $\bigcap_{i=1}^m Q_i\neq\{0\}$.
\end{obs}

\begin{example}\label{conj indef pero no strongly}
In what follows we give an example of a maximal indefinite set.
Consider the operators $B_1,\ldots,B_4$ acting on $\CC^4$ which are represented by
\begin{align*}
B_1=\begin{bmatrix}
1& 0& 0& 0\\
0& -1& 0& 0\\
0&0&0&1\\
0&0&1&0
\end{bmatrix},\quad
B_2=\begin{bmatrix}
0&1&0&0\\
1&0&0&0\\
0&0&1&0\\
0&0&0&-1
\end{bmatrix},\\
B_3=\begin{bmatrix}
1&0&0&0\\
0&0&1&0\\
0&1&0&0\\
0&0&0&-1
\end{bmatrix},\quad
B_4=\begin{bmatrix}
0&0&0&1\\
0&1&0&0\\
0&0&-1&0\\
1&0&0&0
\end{bmatrix}.
\end{align*}
These four matrices satisfy that $B_j$ is indefinite in $\bigcap_{i\neq j}Q_i$ for $j=1,2,3,4$. Indeed,
\begin{eqnarray*}
{\scriptstyle(1-\sqrt{2},-1,1-\sqrt{2},1)}\in \mc{P}_1^-\cap \bigcap_{i\neq 1}Q_i, \ & \ {\scriptstyle(3-\sqrt{2},1+2\sqrt{2},-1+5\sqrt{2},7)}\in \mc{P}_1^+\cap \bigcap_{i\neq 1}Q_i, \\
{\scriptstyle(1-\sqrt{2},1,-1+\sqrt{2},1)}\in \mc{P}_2^-\cap\bigcap_{i\neq 2}Q_i, \ & \  {\scriptstyle(1+\sqrt{2},1,-1-\sqrt{2},1)}\in \mc{P}_2^+\cap \bigcap_{i\neq 2}Q_i, \\
{\scriptstyle(1-\sqrt{2},-1,-1+\sqrt{2},1)}\in \mc{P}_3^-\cap\bigcap_{i\neq 3}Q_i, \ & \  {\scriptstyle(1-\sqrt{2},-1+2\sqrt{2},3-\sqrt{2},1)}\in \mc{P}_3^+\cap\bigcap_{i\neq 3}Q_i, \\
{\scriptstyle(1-5\sqrt{2},-1-2\sqrt{2},-3+\sqrt{2},7)}\in \mc{P}_4^-\cap\bigcap_{i\neq 4}Q_i, \ & \  {\scriptstyle(-1+\sqrt{2},1,-1+\sqrt{2},1)}\in \mc{P}_4^+\cap\bigcap_{i\neq 4}Q_i.
\end{eqnarray*}
Nevertheless, $\bigcap_{i=1}^4 Q_i=\{0\}$ because the system of equations
\begin{align*}
\left\{ 
\begin{array}{rcl}
|x_1|^2+|x_2|^2+2\real(x_3x_4) &=& 0 \\ 
\,2\real(x_1x_2)+|x_3|^2+|x_4|^2 &=& 0\\
|x_1|^2+2\real(x_2x_3)+|x_4|^2 &=& 0\\
2\real(x_1x_4)+|x_2|^2+|x_3|^2 &=&0 
\end{array}
\right.
\end{align*}
admits only the trivial solution.

\end{example}

\begin{lema}\label{lema_opciones} 
Let $\{B_1,\ldots,B_m\}$ be an indefinite set.
Take $y\in\mc{P}^-_k\cap\bigcap_{i\neq k}Q_i$ and $z\in\mc{P}^+_k\cap\bigcap_{i\neq k}Q_i$ for some $k=1,\ldots,m$ and consider $\mc{S}:=\{\alpha y+\beta z\,:\,\alpha,\beta\in\RR\}$. Then,
\[
\text{either}\qquad\mc{S}\subseteq\textstyle{\bigcap_{i\neq k}Q_i}\qquad\text{or}\qquad\mc{S}\cap Q_k\cap Q_l=\{0\}\qquad\text{for some $l\neq k$}.
\]
\end{lema}

\begin{proof}
Assume that there exists $l\neq k$ such that
$\mc{S}\nsubseteq Q_l$. 
This is equivalent to $\real\PI{B_ly}{z}\neq0$.
Then for any $\alpha,\beta\in\RR$, $\alpha\neq0$, $\beta\neq0$,
$\alpha y+\beta z\notin Q_l$. Since $y,z\notin Q_k$ we get that $\mc{S}\cap Q_k\cap Q_l=\{0\}$.
	\end{proof}

The following lemma shows that, under suitable hypotheses, proving that the notions of indefinite and weakly indefinite sets coincide is equivalent to generalizing Krein-\v{S}mul'jan theorem.

\begin{lema}\label{problemas equiv}
Given $B_1,\ldots,B_m\in \mc{L}(\HH)$, assume that $\{B_1,\ldots,B_m\}$ is linearly independent, $B_j\not\equiv0$ in $\bigcap_{i\neq j} Q_i$ for every $j=1,\ldots,m$, and  $\{B_j\}_{j\in J}$ is weakly indefinite for every $J\subset\{1,\ldots,m\}$ with $|J|=m-1$. Then, the following statements are equivalent:
\begin{enumerate}
	\item[i)] there exists $j=1,\ldots,m$ such that $B_j\geq 0$	in $\bigcap_{i\neq j} Q_i$ if and only if $B_j + \sum_{i\neq j} \la_iB_i\geq 0$ for some $(\la_j)_{j\neq i}\in\RR^{m-1}$;
	
	
	\item[ii)] $\{B_1,\ldots,B_m\}$ is indefinite if and only if $\{B_1,\ldots,B_m\}$ is weakly indefinite.
\end{enumerate}
\end{lema}

\begin{proof}
%
Assume that $i)$ holds and also that $\{B_1,\ldots,B_m\}$ is not indefinite, i.e. there exists $j=1,\ldots,m$ such that $B_j$ is semidefinite in $\bigcap_{i\neq j} Q_i$. If $B_j\geq 0$ in $\bigcap_{i\neq j} Q_i$ then, by $i)$, there exists $(\la_i)_{i\neq j}\in\RR^{m-1}$ such that $B_j + \sum_{i\neq j} \la_iB_i\geq 0$. If $B_j\leq 0$ in $\bigcap_{i\neq j} Q_i$ then $-B_j\geq 0$ in $\bigcap_{i\neq j} Q_i$ and, by $i)$, there exists $(\mu_i)_{i\neq j}\in\RR^{m-1}$ such that $-B_j + \sum_{i\neq j} \mu_iB_i\geq 0$.  Therefore, $\{B_1,\ldots,B_m\}$ is not weakly indefinite. The converse implication is always true, see Lemma \ref{muy indef implica indef}. Thus, $\rm i)$ implies $\rm ii)$. 


Conversely, assume that $ii)$ holds and also that there exists $j=1,\ldots,m$ such that $B_j\geq 0$ in $\bigcap_{i\neq j} Q_i$. Then, by $ii)$, $\{B_1,\ldots,B_m\}$ is neither indefinite nor weakly indefinite. Hence, there exists $(\la_i)_{i\neq j}\in\RR^{m-1}$ such that $B_j + \sum_{i\neq j} \la_iB_i$ is semidefinite. Since $B_j\not\equiv0$ in $\bigcap_{i\neq j} Q_i$, there exists $x\in \bigcap_{i\neq j} Q_i$ such that $\PI{B_jx}{x}>0$. Hence,
\[
\PI{\big(B_j + \sum_{i\neq j} \la_iB_i\big)x}{x}=\PI{B_j x}{x}>0,
\]
which proves that $B_j + \sum_{i\neq j} \la_iB_i\geq 0$. The converse implication is immediate. Therefore, $ii)$ implies $i)$.
\end{proof}


\section{Strongly indefinite sets}

Given a set $\{B_1,\ldots,B_m\}$ of selfadjoint operators, our aim is to impose condition(s) onto it in order to prove a generalization of Krein-Smul'jan theorem of the form: if $A\in\mc{L}(\HH)$ is selfadjoint then,
\begin{equation}\label{KSm}
A\geq0\quad\text{in}\quad \textstyle{\bigcap_{i=1}^m}Q_i\qquad\text{if and only if}\qquad\Omega\neq\varnothing.
\end{equation}

\begin{obs}
If \eqref{KSm} holds for a weakly indefinite set $\{B_1,\ldots,B_m\}$ then $\bigcap_{i=1}^mQ_i\neq \{0\}$. 

Indeed, given a weakly indefinite set $\{B_1,\ldots,B_m\}$ (where $m$ is less than the dimension of the real subspace of selfadjoint operators in $\HH$) suppose that $\bigcap_{i=1}^mQ_i= \{0\}$ and consider any selfadjoint operator $B\in\mc{L}(\HH)$ such that $\{B_1,\ldots,B_m, B\}$ is a linearly independent set. Since both $B\geq 0$ and $-B\geq 0$ in  $\bigcap_{i=1}^mQ_i= \{0\}$, if \eqref{KSm} holds then there exist $(\la_1,\ldots,\la_m),(\mu_1,\ldots,\mu_m)\in\RR^m\setminus\{0\}$ such that 
\[
B + \sum_{i=1}^m \la_i B_i \geq 0 \qquad \text{and} \qquad -B + \sum_{i=1}^m \mu_i B_i\geq 0.
\]
Then, there exists $j=1,\ldots,m$ such that $\mu_j\neq -\la_j$, otherwise, $B+ \sum_{i=1}^m \la_i B_i=0$ which is a contradiction to  $\{B_1,\ldots,B_m, B\}$ being linearly independent.

Hence, adding the above inequalities we get that $\sum_{i=1}^m (\la_i + \mu_i)B_i\geq 0$ which is a contradiction to $\{B_1,\ldots,B_m\}$ being a weakly indefinite set.
\end{obs}

Example \ref{conj indef pero no strongly} presents an indefinite set $\{B_1,B_2,B_3,B_4\}$ such that $\bigcap_{i=1}^4 Q_i=\{0\}$. Hence, for $m\geq 3$ assuming that $\{B_1,\ldots,B_m\}$ is weakly indefinite, or even indefinite, is not enough as a suitable hypothesis for generalizing Theorem \ref{teo_krein_smuljan_azizov}. 

\medskip

If $\{B_1,\ldots,B_m\}$ is a weakly indefinite set with $m\geq 3$ then, by Corollary \ref{prop_simetria_2}, any trio $\{B_i,B_j,B_k\}$ is an indefinite set. In particular $B_i$ is indefinite in $Q_j\cap Q_k$, i.e. there always exist $x_+\in\mc{P}_i^+\cap Q_j\cap Q_k$ and $x_-\in\mc{P}_i^-\cap Q_j\cap Q_k$.

Also, by Lemma \ref{lema_opciones}, if $x_\pm\in \mc{P}_i^\pm\cap Q_j\cap Q_k$ and $\mc{S}:=\{\alpha x_-+\beta x_+:\,\alpha,\beta\in\RR\}$ then either $\mc{S}\subseteq Q_j\cap Q_k$ or $\mc{S}\cap Q_i\cap Q_j=\{0\}$ or $\mc{S}\cap Q_i\cap Q_k=\{0\}$.


\begin{Def} \label{strongly indef}
Given selfadjoint operators $B_1,\ldots,B_m\in\mc{L}(\HH)$, $m\geq2$, the set $\{B_1,\ldots,B_m\}$ is {\it strongly indefinite} if
\begin{enumerate}[label=\roman*)]
\item $\{B_1,\ldots,B_m\}$ is weakly indefinite;
\item given $i,j,k=1,\ldots,m$, if $x_\pm\in \mc{P}_i^\pm\cap Q_j\cap Q_k$ then there exists $\theta\in[0,\pi)$ such that
$B_j$ and $B_k$ are definite in $\{\alpha x_-+\beta e^{i\theta}x_+:\alpha,\beta\in\RR\}$.
\end{enumerate}
\end{Def}

\medskip

Given a selfadjoint operator $B\in\mc{L}(\HH)$, assume that $\{y,z\}$ is a linearly independent set in $Q(B)$. Then, $B$ is definite in $\{\alpha y+\beta z:\alpha,\beta\in\RR\}$ if and only
if $\real\PI{By}{z}=0$. In fact, if
$x_\pm=ty\pm(1-t)z$ for some $t\in \RR$ then 
$$
\PI{Bx_\pm}{x_\pm}= \pm2t(1-t)\real\PI{By}{z},
$$ 
and these two real numbers have the same sign if and only if $\real\PI{By}{z}=0$. Therefore, item $\rm ii)$ in Definition \ref{strongly indef} can be alternatively stated as:
\begin{align*}
   \text{{\rm ii')}} \  & \text{given $i,j,k=1,\ldots,m$, \ if $x_\pm\in \mc{P}_i^\pm\cap Q_j\cap Q_k$ then there exists}\ \theta\in[0,\pi)  \\ &  \text{such that $\real\PI{B_j x_+}{e^{i\theta}x_-}=\real\PI{B_k x_+}{e^{i\theta}x_-}=0$.}
\end{align*}

\begin{obs}

If $\{B_1,\ldots,B_m\}$ is strongly indefinite, then it is immediate that $\{B_i\}_{i\in\mc{F}}$ is strongly indefinite for every $\mc{F}\subseteq\{1,\ldots,m\}$.
\end{obs}

\smallskip

From now on we assume that $m> 2$. Given $y,z\in\HH$, $y\neq z$, consider
\[
[y,z]:= \big\{\,ty+(1-t)z\,:\,t\in[0,1]\big\},
\]
and $(y,z):= \big\{\,ty+(1-t)z\,:\,t\in(0,1)\big\}$.

\begin{prop}\label{lema_horrible} 
Let $\{B_1,\ldots,B_m\}$ be a strongly indefinite set. Given $i=1,\ldots,m$, if there exists $x_\pm\in\mc{P}_i^\pm\cap\bigcap_{j\neq i}Q_j$ then there exists $\theta\in[0,\pi)$ such that
\[
[x_-,\pm e^{i\theta}x_+]\subseteq\textstyle{\bigcap_{j\neq i}}Q_j.
\]
Moreover, there exists $y_i^\pm\in (x_-, \pm e^{i\theta}x_+)$ such that
\begin{equation}
y_i^\pm \in \bigcap_{i=1}^m Q_i\setminus N\big(\textstyle{\sum_{j=1}^m \mu_j B_j}\big)
\end{equation}
for every $(\mu_1,\ldots,\mu_m)\in \RR^{m}$ with $\mu_i\neq 0$.
\end{prop}

\begin{proof}
Suppose that 
$x_\pm\in\mc{P}_i^\pm\cap\bigcap_{l\neq i}Q_l$ for some fixed $i=1,\ldots,m$. Now choose $j\in\{1,\ldots,m\}\setminus\{ i\}$. Considering $k_1,k_2\in\{1,\ldots,m\}\setminus\{ i\}$,
there exist $\theta_1,\theta_2\in[0,\pi)$
such that
\begin{align*}
\real\big(e^{-i\theta_1}\PI{B_jx_-}{x_+}\big)=& \ 0=\real\big(e^{-i\theta_1}\PI{B_{k_1}x_-}{x_+}\big), \ \text{and} \\
\real\big(e^{-i\theta_2}\PI{B_jx_-}{x_+}\big)=& \ 0=\real\big(e^{-i\theta_2}\PI{B_{k_2}x_-}{x_+}\big).
\end{align*}
This implies that $\theta_2=\theta_1+n\pi$ for some $n\in\NN$, and consequently
\begin{align*}
\real\PI{B_{k_2}x_-}{e^{i\theta_1}x_+} &=\pm\real\PI{B_{k_2}x_-}{e^{i\theta_2}x_+}=0.
\end{align*}
Since $k_2$ was arbitrary, it then holds that
\[
\real\PI{B_kx_-}{e^{i\theta_1}x_+}=0 \qquad\text{for every $k\neq i$}.
\]
Therefore, $[x_-,\pm e^{i\theta}x_+]\subseteq Q_k$ for every $k\neq i$ because $x_\pm\in\bigcap_{k\neq i}Q_k$.

Finally, following the same procedure as in the proof of Proposition \ref{prop_simetria}, 
there exists $t\in (0,1)$ such that
\[
y_i^+:=tx_-+(1-t)e^{i\theta}x_+\in \textstyle{\bigcap_{i=1}^m}Q_i\setminus\{0\}.
\]
Given $(\mu_1,\ldots,\mu_m)\in\RR^m$ with $\mu_i\neq 0$, 
consider $B:=B_i+\sum_{j\neq i}\tfrac{\mu_j}{\mu_i} B_j$ and assume that $By_i^+=0$.
Then, $Bx_-=-\frac{1-t}{t}{e^{i\theta}}Bx_+$. 
Since $x_\pm\in\bigcap_{j\neq i}Q_i$, we have that
\begin{align*}
0=\PI{By_i^+}{y_i^+}&=t^2\PI{Bx_-}{x_-}+(1-t)^2\PI{Bx_+}{x_+}+2t(1-t)\real\PI{Bx_-}{e^{i\theta}x_+}\\
&=t^2\PI{B x_-}{x_-}+(1-t)^2\PI{Bx_+}{x_+}-2(1-t)^2\real\PI{Bx_+}{x_+}\\
&=t^2\PI{B_ix_-}{x_-}-(1-t)^2\real\PI{B_ix_+}{x_+}<0,
\end{align*}
leading to a contradiction. Therefore, $y_i^+\notin N(B)$. A similar argument proves the existence of $y_i^-$.
\end{proof}

\begin{obs}
By the proof of Theorem \ref{prop_krein_smuljan_2}, every weakly indefinite set $\{B_1,B_2\}$ is strongly indefinite. Also, if $\{B_1,\ldots, B_m\}$ is a strongly indefinite set and there exists $j=1,\ldots,m$ such that $B_j$ is indefinite in $\bigcap_{i\neq j} Q_i$, then $\bigcap_{i=1}^m Q_i\neq\{0\}$ (see Proposition \ref{lema_horrible}). Hence, $\{B_1,B_2,B_3,B_4\}$ from Example \ref{conj indef pero no strongly} is an indefinite set which is not strongly indefinite.
\end{obs}

%

The following result generalizes Krein-Smul'jan theorem for $m\geq 3$.

\begin{thm}\label{prop_ks_strongly} 
Given selfadjoint operators $A,B_1,\ldots,B_m\in\mc{L}(\HH)$, assume that
$\{B_1,B_2,\ldots,B_m\}$ is strongly indefinite. Then,
\[
A\geq0\quad\text{in}\quad\textstyle{\bigcap_{i=1}^m}Q_i\qquad\text{if and only if}\qquad\Omega\neq\varnothing.
\]
\end{thm}

\begin{proof}
The fact that $\Omega\neq\varnothing$ implies $A\geq0$ in $\bigcap_{i=1}^mQ_i$ is trivial. We prove the converse by induction on $m$. The case for $m=2$ operators $B_1,B_2$ follows readily from Theorem
\ref{prop_krein_smuljan_2}.

For the inductive step fix $n\in\NN$, $n\geq3$, and assume the statement holds for $m=n-1$. Now consider
selfadjoint operators $B_1,B_2,\ldots,B_n\in\mc{L}(\HH)$ such that $\{B_1,B_2,\ldots,B_n\}$ is strongly indefinite.

First, let us show that $B_j\not\equiv 0$ in $\bigcap_{i\neq j} Q_i$ for every $j=1,\ldots,n$. Indeed, if there exists $j=1,\ldots,n$ such that $B_j\equiv 0$ in $\bigcap_{i\neq j} Q_i$ then, by inductive hypothesis, there exists $(\la_i)_{i\neq j}\in \RR^{n-1}$ such that $B_j + \sum_{i\neq j} \la_i B_i\geq 0$, which is a contradiction. Then
 $\{B_1,\ldots,B_n\}$ is indefinite (by Remark \ref{problemas equiv}) and, by Proposition \ref{lema_horrible}, $\bigcap_{i=1}^{n}Q_i\neq\{0\}$. 


Now, assume that $A\geq0$ in $\bigcap_{i=1}^{n}Q_i$. Since $B_n$ is indefinite in $\bigcap_{i=1}^{n-1}Q_i$, take $y\in\mc{P}_{n}^-\cap \bigcap_{i=1}^{n-1}Q_i$ and $z\in\mc{P}_{n}^+\cap\bigcap_{i=1}^{n-1}Q_i$. Again, by Proposition \ref{lema_horrible}, there exist $\theta\in[0,\pi)$ and $t_\pm\in(0,1)$ such that
$x_\pm:=t_\pm y\pm(1-t_\pm)e^{i\theta}z\in\bigcap_{i=1}^nQ_i$. Then $\PI{Ax_\pm}{x_\pm}\geq0$.

Following a procedure similar to the one in the proof of Theorem \ref{prop_krein_smuljan_2} we then get that
\[
\frac{\PI{Ay}{y}}{\PI{B_n y}{y}}\leq\frac{\PI{Az}{z}}{\PI{B_nz}{z}},
\]
for arbitrary $y\in\mc{P}_{n}^-\cap \bigcap_{i=1}^{n-1}Q_i$ and $z\in\mc{P}_{n}^+\cap\bigcap_{i=1}^{n-1}Q_i$. Hence, there exists $\la_n\in\RR$ such that
\[
\PI{(A+\la_n B_n)x}{x}\geq0 \quad\text{for every $x\in\big(\mc{P}_n^-\cap \textstyle{\bigcap_{i=1}^{n-1}}Q_i\big)\cup \big(\mc{P}_n^+\cap \textstyle{\bigcap_{i=1}^{n-1}}Q_i\big)$.}
\]
Considering that $\PI{Ax}{x}\geq0$ for every $x\in \bigcap_{i=1}^nQ_i$, we have that
\[
\PI{(A+\la_n B_n)x}{x}\geq0\qquad\text{for every $x\in \textstyle{\bigcap_{i=1}^{n-1}}Q_i$}.
\]
Then, applying  the inductive hypothesis to $A':=A+\la_n B_n$ and the strongly indefinite set $\{B_1,B_2,\ldots,B_{n-1}\}$, there exists $(\la_1,\la_2,\ldots,\la_{n-1})\in\RR^{n-1}$ such that 
\[
(A+\la_n B_n)+\sum_{i=1}^{n-1}\la_iB_i\geq0, 
\]i.e.
$(\la_1,\la_2,\ldots,\la_n)\in\Omega$, completing the proof.
\end{proof}

\begin{cor}\label{strongly implica indef} 
Given selfadjoint operators $B_1,\ldots,B_m\in\mc{L}(\HH)$, if $\{B_1,\ldots,B_m\}$ is strongly indefinite then $\{B_1,\ldots,B_m\}$ is indefinite.
\end{cor}

\begin{proof}
Suppose that $\{B_1,\ldots,B_m\}$ is strongly indefinite and there exists $i=1,\ldots,m$ such that $B_i$ is definite in $\bigcap_{j\neq i} Q_j$. Let us assume that $B_i\geq 0$ in $\bigcap_{j\neq i} Q_j$. Note that $\{B_1,\ldots,B_{i-1},B_{i+1},\ldots,B_m\}$ is also strongly indefinite and, by Theorem \ref{prop_ks_strongly}, there exists $(\la_j)_{j\neq i}\in\RR^{m-1}$ such that $B_i + \sum_{j\neq i} \la_jB_j \geq 0$, which is a contradiction to $\{B_1,\ldots,B_m\}$ being weakly indefinite.
\end{proof}

\begin{cor}
Given selfadjoint operators $B_1,\ldots,B_m\in\mc{L}(\HH)$, the following conditions are equivalent:
\begin{enumerate}[label=\roman*)]
\item $\{B_1,\ldots,B_m\}$ is strongly indefinite.
\item \begin{enumerate}
\item $\{B_1,\ldots,B_m\}$ is indefinite;
\item if $x_\pm\in \mc{P}_i^\pm\cap Q_j\cap Q_k$ then there exists $\theta\in[0,\pi)$ such that 
\[
\real\PI{B_j x_+}{e^{i\theta}x_-}=\real\PI{B_k x_+}{e^{i\theta}x_-}=0.
\]
\end{enumerate}
\item \begin{enumerate}
\item for each $i=1,\ldots,m$ there exists $x_i \in \textstyle{\bigcap_{j=1}^m}Q_j\setminus N(\sum_{j=1}^m \mu_j B_j)$ for every choice of $(\mu_1,\ldots,,\mu_m)\in\RR^{m}$ with $\mu_i\neq 0$;
\item  if $x_\pm\in\mc{P}_i^\pm\cap Q_j\cap Q_k$ then there exist $\theta\in[0,\pi)$ and $x_\theta\in(x_-,e^{i\theta}x_+)$ such that
\[
x_\theta\in Q_i\cap Q_j\cap Q_k.
\]
\end{enumerate}
\end{enumerate}
\end{cor}

\begin{proof}
\noi {\it i)$\to$ii)}\  If $\{B_1,\ldots,B_m\}$ is strongly indefinite then, by Corollary \ref{strongly implica indef}, the set $\{B_1,\ldots,B_m\}$ is indefinite. 
We have already mentioned that {\it(b)} is equivalent to the second condition in Definition \ref{strongly indef}.

%

\smallskip

\noi {\it ii)$\to$iii)}\  Item $(a)$ follows from the fact that $\{B_1,\ldots,B_m\}$ is indefinite and Proposition \ref{lema_horrible}. Fix $i,j,k\in\{1,2,\ldots,m\}$ and take $x_\pm\in\mc{P}_i^\pm\cap Q_j\cap Q_k$.
Since $\{B_1,\ldots,B_m\}$ is indefinite, $\{B_i,B_j,B_k\}$ is also indefinite, and the result follows from Proposition \ref{lema_horrible}.

\smallskip

\noi {\it iii)$\to$i)}\  To see that $\{B_1,\ldots,B_m\}$ is weakly indefinite, consider $(\mu_1,\ldots,\mu_m)\in\RR^m$ 
and suppose that $\mu_i\neq 0$ for some $i=1,\ldots,m$. Then, by $(a)$, there exists $x_i\in \textstyle{\bigcap_{j=1}^m}Q_j\setminus N(\sum_{j=1}^m \mu_j B_j)$. Hence, $x_i\in Q(\sum_{j=1}^m \mu_j B_j)\setminus N(\sum_{j=1}^m \mu_j B_j)$, which implies that $\sum_{j=1}^m \mu_j B_j$ is indefinite. Since $(\mu_1,\ldots,\mu_m)\in\RR^m\setminus\{0\}$ was arbitrary, we have that $\{B_1,\ldots,B_m\}$ is weakly indefinite.

Fix $i,j,k\in\{1,2,\ldots,m\}$ and take $x_\pm\in\mc{P}_i^\pm\cap Q_j\cap Q_k$. By $iii)$ there exist $\theta\in[0,\pi)$
and $t_0\in(0,1)$ such that $x_\theta:=t_0x_-+(1-t_0)e^{i\theta}x_+\in Q_i\cap Q_k \cap Q_k$. This implies that
$\real\PI{B_jx_-}{e^{i\theta}x_+}=\real\PI{B_kx_-}{e^{i\theta}x_+}=0$, which in turn implies that $B_j$ and $B_k$ are definite in $\{\alpha x_-+\beta e^{i\theta}x_+:\alpha,\beta\in\RR\}$.
\end{proof}

\section{A sufficient condition for strongly indefiniteness} \label{Hes-McShane}

Given a set of selfadjoint operators $\{B_1,\ldots,B_m\}$ in $\mc{L}(\HH)$, we now present a sufficient condition to guarantee that it is a strongly indefinite set. It is inspired by previous works by Hestenes and McShane for the (real) finite dimensional case. Given symmetric matrices $A,B_1,\ldots,B_m\in\RR^{n\x n}$, assume that $\{B_1,\ldots,B_m\}$ is weakly indefinite. In \cite{Hestenes} the authors included the following additional condition: for every subspace $L$ of $\RR^n$ such that $L\cap\left(\bigcap_{i=1}^m Q_i\right)=\{0\}$ there exists $(\mu_1,\ldots,\mu_m)\in\RR^m\setminus\{0\}$ such that $\sum_{i=1}^m\mu_i B_i$ is positive definite in the subspace $L$. Under these assumptions they showed that, if $A$ is positive definite in $\bigcap_{i=1}^m Q_i$ then there exists $(\la_1,\ldots,\la_m)\in\RR^m$ such that $A+ \sum_{i=1}^m \la_i B_i$ is positive definite.

\medskip

\begin{hyp}[\bf{HM}]
Given $m\geq3$ and $B_1,\ldots,B_m\in\mc{L}(\HH)$, assume that the set
$\{B_1,\ldots,B_m\}$ is weakly indefinite. Assume also that if $\mc{S}$ is a real subspace of $\HH$ with $\dim\mc{S}=2$ and $\{i,j,k\}\subset \{1,\ldots,m\}$ is a trio such that
\[
\mc{S}\cap(Q_i\cap Q_j\cap Q_k)=\{0\} 
\]
then there exists $(\la_i,\la_j,\la_k)\in\RR^3\setminus\{0\}$ such that $\la_iB_i+\la_jB_j+\la_kB_k\geq 0$ in $\mc{S}$ and $\la_iB_i+\la_jB_j+\la_kB_k\neq 0$ in $\mc{S}$.
\end{hyp}

If the set $\{B_1,\ldots,B_m\}$ satisfies Hypotheses (HM) then it is immediate that $\{B_i\}_{i\in\mc{F}}$ also satisfies Hypotheses (HM)  for every  $\mc{F}\subseteq\{1,\ldots,m\}$ with $|\mc{F}|\geq 3$.

\begin{prop}\label{prop_hestenes_2} 
Given $m\geq3$ and $B_1,\ldots,B_m\in\mc{L}(\HH)$, assume that $\{B_1,\ldots,B_m\}$ satisfies Hypotheses (HM).
Then, $\{B_1,\ldots,B_m\}$ is a strongly indefinite set.
\end{prop}

\begin{proof}
We prove the result by induction on $m$. First, assume that $m=3$. By Corollary \ref{prop_simetria_2},  $\{B_1,B_2,B_3\}$ is indefinite. Fix $i\in\{1,2,3\}$ and consider $x_\pm\in\mc{P}_i^\pm\cap Q_j\cap Q_k$. On the one hand, if $k=j$ then $x_\pm\in\mc{P}_i^\pm\cap Q_j$ and we can choose $\theta\in[0,\pi)$ such that $\real\PI{B_jx_+}{e^{i\theta}x_-}=0$. Hence, in this case we have that $B_j=B_k$ is zero in $\St=\{\alpha x_+ + \beta e^{i\theta}x_-: \ \alpha,\beta\in\RR\}$.


On the other hand, if $k\neq j$ choose $\theta\in[0,\pi)$ such that $\real\PI{B_j x_+}{e^{i\theta}x_-}=0$ and consider the real subspace 
$$\mc{S}=\{\alpha x_+ + e^{i\theta}\beta x_-\,:\,\alpha,\beta\in\RR\}\subseteq Q_j.$$ 
If $\mc{S}\cap Q_i\cap Q_j\cap Q_k=\set{0}$, then there exist $\la_i,\la_j,\la_k\in\RR$ such that $B:=\la_iB_i+\la_jB_j+\la_kB_k\geq0$ (and non zero) in $\mc{S}$.
Thus,
\[
0\leq\PI{Bx_}{x_-}=\la_i\PI{B_i x_-}{x_-}\qquad\text{and}\qquad0\leq\PI{Bx_+}{x_+}=\la_i\PI{B_i x_+}{x_+}.
\]
But $\PI{B_i x_-}{x_-}<0$ and $\PI{B_i x_+}{x_+}>0$ implies that $\la_i=0$. Since $B_j\big|_\mc{S}=0$ we get that $B\big|_{\mc{S}}=\la_kB_k\big|_{\mc{S}}$. Then
\begin{align*}
0\leq &\PI{B(x_+ +e^{i\theta}x_-)}{x_+ +e^{i\theta}x_-}=2\la_k\real\PI{B_k x_+}{e^{i\theta}x_-}, 
\\
0\leq &\PI{B(x_+ -e^{i\theta}x_-)}{x_+ -e^{i\theta}x_-}=-2\la_k\real\PI{B_k x_+}{e^{i\theta}x_-},
\end{align*}
and consequently either $\la_k=0$ or $\real\PI{B_k x_+}{e^{i\theta}x_-}=0$.
But if $\la_k=0$ then $B\big|_\mc{S}=0$, leading to a contradiction. Hence, $\real\PI{B_k x_+}{e^{i\theta} x_-}=0$ and condition \rm{ii')} is verified. Therefore, $\{B_1,B_2,B_3\}$ is strongly indefinite.

For the inductive step fix $n\in\NN$, $n\geq4$, and assume the statement holds for $m=n-1$ operators. Now consider
$B_1,\ldots,B_n\in\mc{L}(\HH)$ satisfying the hypotheses.

Hence, by inductive hypothesis, $\{B_1,\ldots,B_{k-1},B_{k+1},\ldots,B_n\}$ is strongly indefinite for $k=1,\ldots,n$. Then, by Remark \ref{problemas equiv}, $\{B_1,\ldots,B_n\}$ is indefinite.

Now take three different indices $i,j,k\in\{1,\ldots,n\}$. If none of them is equal to $n$, by inductive hypothesis, item \rm{ii)} in the definition of strongly indefiniteness is satisfied. Assume that $k=n$ and take
$y\in\mc{P}_{n}^-\cap \bigcap_{l=1}^{n-1}Q_l$, $z\in\mc{P}_{n}^+\cap\bigcap_{l=1}^{n-1}Q_l$. Then, choose $\theta\in[0,\pi)$ such that $\real\PI{B_iy}{e^{i\theta}z}=0$ and consider the real subspace $\mc{S}=\{\alpha y+e^{i\theta}\beta z\,:\,\alpha,\beta\in\RR\}\subseteq Q_i$. 
If $\mc{S}\cap Q_i\cap Q_j\cap Q_n=\set{0}$ then, following the same procedure as in the previous step,
$\real\PI{B_jy}{e^{i\theta} z}=0$ and condition \rm{ii')} is verified. 
Therefore, $\{B_1,\ldots,B_n\}$ is strongly indefinite.
\end{proof}

\begin{cor}
Given $B_1,\ldots,B_m\in\mc{L}(\HH)$ assume that
$\{B_1,\ldots,B_m\}$ satisfies Hypotheses (HM). If $A\in\mc{L}(\HH)$ is selfadjoint, then
\[
A\geq0\quad\text{in}\quad\textstyle{\bigcap_{i=1}^m}Q_i\qquad\text{if and only if}\qquad\Omega\neq\varnothing.
\]
\end{cor}

\section*{Acknowledgements}

The authors gratefully acknowledge the support of CONICET through the grant PIP 11220200102127CO. F.~Mart\'{\i}nez Per\'{\i}a also acknowledges the support from UNLP 11X974. This research was partially supported by the Air Force Office of Scientific Research (USA) grant FA9550-24-1-0433.

\end{document}